\documentclass[letterpaper, 10 pt, conference]{ieeeconf}

\usepackage{amsmath}
\usepackage{amsthm}
\usepackage{amssymb}
\usepackage{graphicx}
\usepackage{url}
\usepackage{bbm}
\usepackage{epsfig}
\usepackage{textcomp}
\usepackage{multirow}
\usepackage{subfig}

\IEEEoverridecommandlockouts
\newtheorem{rem}{Remark}
\newtheorem{cor}{Corollary}
\newtheorem{lem}{Lemma}
\newtheorem{thm}{Theorem}

\newtheorem{prop}{Proposition}

\makeatletter
\providecommand{\tabularnewline}{\\}
\@ifundefined{showcaptionsetup}{}{%
 \PassOptionsToPackage{caption=false}{subfig}}
\usepackage{subfig}
\makeatother

\begin{document}

\title{\LARGE \bf
Stability Analysis of Multi-Period Electricity Market with Heterogeneous
Dynamic Assets
}

\author{Lin Zhao and Wei Zhang
\thanks{
This work was supported (in part) by the Control of Complex Systems Initiative, a Laboratory Directed Research and Development (LDRD) program at the Pacific Northwest National Laboratory.
\newline
\indent L. Zhao is with the Department of Electrical and Computer Engineering, The Ohio State University, Columbus, OH 43210, USA 
{\small (email: zhao.833@osu.edu)}
\newline 
\indent W. Zhang is with the Department of Electrical and Computer Engineering, The Ohio State University, Columbus, OH 43210, USA, with a joint appointment in the Electricity Infrastructure Group, Pacific Northwest National Laboratory, Richland, WA 99354, USA
{\small (email: zhang.491@osu.edu)}}
}
\maketitle
\begin{abstract}
Market-based coordination of demand side assets has gained great interests
in recent years. In spite of its efficiency, there is a risk that
the interaction between the dynamic assets through the price signal
could result in an unstable closed-loop system. This may cause oscillating
power consumption profiles and high volatile energy price. This paper
proposes an electricity market model which explicitly considers the
heterogeneous dynamic asset models. We show that the market dynamics
can be modeled by a discrete nonlinear system, and then derive analytical
conditions to guarantee the stability of the market via contraction
analysis. These conditions imply that the market stability can be
guaranteed by choosing bidding functions with relatively shallower
slopes in the linear region. Finally, numerical examples are provided
to demonstrate the application of the derived stability conditions.
\end{abstract}

\section{Introduction}

To adapt to the distributed generation and increasing penetration
of volatile renewable energies and to improve the overall efficiency
of the power industry, the power systems have undergone a substantial
change from centrally controlled, vertically integrated organizations
to decentrally controlled, deregulated systems. Market mechanism has
been introduced at various levels of power systems to create competition.
In particular, a local electricity market can be created to motivate
self-interested distributed energy resources (DERs) to realize efficient
energy allocation and achieve system-level objectives. Several demonstration
projects have been implemented to validate the idea which showed promising
practice. For example, the GridWise\textregistered\  demonstration
project by the Pacific Northwest National Laboratory showed that the
market-based coordination of residential loads could reduce the utility
demand and congestion at key times~\cite{FullerSchneiderChassin2011}.
The AEP Ohio demonstration project~\cite{AEP} further showed the
capability of enforcing both system-wide constraints and local constraints,
while optimizing both system and individual objectives~\cite{WidergrenSubbaraoFullerEtAl2014}.

Motivated by these projects, recent studies have been focused on the
market mechanism design for engaging different types of DERs~\cite{li2015,HaoCorbinKalsiEtAl2017,BejestaniAnnaswamySamad2014,LiChenLow2011}.
For example, the authors in~\cite{li2015} proposed a market mechanism
to coordinate a population of thermostatically controlled loads (TCLs)
for demand response. The proposed bidding strategies incorporated
the TCL dynamics in order to improve the accuracy and efficiency of
the load coordination. For commercial HVACs (heating, ventilation
and air conditioning), a double-auction market structure was designed
which takes into account the detailed nonlinear building models~\cite{HaoCorbinKalsiEtAl2017}.
The proposed market was demonstrated to be very efficient at peak
shaving and load shifting services. Other market models for demand
response using batteries or PEVs (plug-in electric vehicles) were
also proposed in~\cite{ChenLiLowEtAl2010,LiChenLow2011}.

Despite the popularity of market-based coordination, there has been
a growing concern about the risk of the instability of the power market.
Under some extreme conditions, the aggregate demand and the energy
price can be unstable or demonstrate high volatility over time. Various
factors that contribute to the instability of market-coordinated TCLs
have been examined in~\cite{NazirHiskens}. Some earlier works abstracted
a simple linear differential equation model to quantify the power
market stability~\cite{Alvarado1997,Alvarado1999,NutaroProtopopescu2009}.
A discrete time nonlinear model based on the marginal cost pricing
mechanism was proposed in~\cite{RoozbehaniDahlehMitter2012}. It
assumed that the demand side did not bid into the market and its utility
function was unknown to the system operator. The focus was therefore
on the market instability caused by the uncertain demand prediction
of the coordinator. Under the same framework of~\cite{RoozbehaniDahlehMitter2012},
the authors in~\cite{ZhouRoozbehaniDahlehEtAl2017} considered a
more realistic dynamic consumption model. It was obtained from solving
an optimal inventory control problem, and the market stability was
found to be related to the ratio between the marginal backlog disutility
and the marginal cost of supply.

The aforementioned works only considered aggregate demand models in
autoregression form depending on the previous consumption history
or price history~\cite{RoozbehaniDahlehMitter2012,ZhouRoozbehaniDahlehEtAl2017}.
However, in order to quantify the aggregate demand variation more
accurately, the internal dynamics and operational limits of individual
DERs must be considered. More importantly, the impact of the population
dynamics of the DERs on the overall market dynamics must be investigated
systematically. These features are nevertheless abstracted away in
the existing literature. 

In this paper, we propose a market model which explicitly considers
the individual dynamics of the DERs. Specifically, we model each DER
by a general constrained linear system. Such models have been widely
employed to describe the dynamics of its internal energy state~\cite{Hao2015,ZhaoHaoZhang2016,ZhaoZhangHaoEtAl2017}.
Moreover, we consider the bidding process of the DERs and model the
bidding functions to be dependent on the energy state. Consequently,
the market is cleared at the competitive equilibrium, which results
in an efficient energy allocation for each individual DERs. These
features distinguish our work from those by~\cite{ZhouRoozbehaniDahlehEtAl2017,RoozbehaniDahlehMitter2012}:
the latter of which assumed no bidding process and the price was ex-ante
which may not clear the market.

Under the proposed market structure with heterogeneous dynamic DER
models, we translate the analysis of the market stability into the
stability analysis of a closed-loop system of the DER dynamics. In
general, this system can be viewed as a model predictive control (MPC)
system~\cite{MayneRawlingsRaoEtAl2000,GrunePannek2011} or systems
with optimization based controller~\cite{HeathWills2005,HeathLi2008,Primbs2001,KordaJones2017}.
By assuming quadratic utility and cost functions, it can be further
shown to be a piecewise linear system~\cite{BemporadMorariDuaEtAl2002}.
Alternatively, it can be viewed as an input saturation system with
state-dependent saturation limits. Most of the existing stability
results considered only special cases of such systems~\cite{MayneRawlingsRaoEtAl2000,HeathWills2005,HeathLi2008,Primbs2001,Liu1992,HuLin2001,HuLin2001a},
which are generally not applicable to the system considered in this
paper. In fact, even the analysis of these much simpler systems are
very challenging. While general numerical stability tests via Lyapunov
methods may be employed, they lend little insights into the market
practice. In addition, they are very conservative and can become intractable
as the number of DERs increases.

To address the above challenges, we propose a contraction analysis
based approach to analyzing the stability of such systems. This approach
enables us to derive analytical conditions that guarantee the market
stability. These conditions provide important insights into the design
of the bidding functions. The key observation is that the market stability
can be always guaranteed by selecting shallower bidding functions,
that is, the linear region of the bidding function should have a relatively
small slope. Moreover, these conditions are very mild, and thus leave
the full freedom to the individual users to design their desired bidding
functions while ensuring a stable electricity market.

The rest of this paper is organized as follows. The market structure
and its equilibrium are described in Section~II. The characterization
of the equilibrium and the stability results are presented in Section
III. Numerical examples illustrating the application of the stability
results is provided in Section IV. Finally, the paper is concluded
in Section IV.

\textbf{Notation}: We use ${\bf 1}_{m}$ to represent the $m$ dimensional
column vector of all ones, and $I_{m}$ the $m$ dimensional identity
matrix. The index set $\{1,2,...,m\}$ will be denoted by the $\mathcal{M}$.
Let $\mathcal{X}$ be a subset of $\mathbb{R}^{m}$, then for $x\in\mathbb{R}^{m}$,
the set $\mathcal{X}-x$ is defined as $\{x'-x,\,\forall x'\in\mathcal{X}\}$.
For a symmetric matrix $S$, the inequality $S\succ0$ ($S\succeq0$
) means that the matrix is positive definite (positive semi-definite).
We denote by $\mathcal{H}_{Q}$ with $Q\succ0$ the Hilbert space
$\mathbb{R}^{m}$ with inner product $\left\langle \cdot,\cdot\right\rangle _{Q}:\,\mathbb{R}^{m}\times\mathbb{R}^{m}\mapsto\mathbb{R}$
defined as $\left\langle x,y\right\rangle _{Q}:=x^{T}Qy$, and induced
norm $\left\Vert \cdot\right\Vert _{Q}:\,\mathbb{R}^{m}\mapsto\mathbb{R}_{\geq0}$
defined as $\left\Vert x\right\Vert _{Q}:=\sqrt{x^{T}Qx}$. The Euclidean
2-norm will be denoted by $\left\Vert \cdot\right\Vert _{2}$, The
projection operator in $\mathcal{H}_{Q}$, denoted by $\text{Proj}_{C}^{Q}(x):\,\mathbb{R}^{m}\mapsto C$,
is defined as $\text{Proj}_{C}^{Q}(x):=\arg\min_{y\in C}\left\Vert x-y\right\Vert _{Q}=\arg\min_{y\in C}\left\Vert x-y\right\Vert _{Q}^{2}$.
When $Q=I_{m}$, we simply write $\text{Proj}_{C}(x)$ which is the
standard projection in $\mathbb{R}^{m}$.

\section{Problem Formulation}

In this section, we will first describe an electricity market model
which involves the bidding and clearing processes at each market period.
The problem of market stability analysis is then defined formally.

\subsection{Market Structure}

Usually a system coordinator is running an electricity market to schedule
and guide the power usage of $m$ DERs. We assume that the $i$th
DER is modeled by the following discrete time scalar linear system
subject to both state and input constraints,
\begin{equation}
x_{i}^{+}=a_{i}x_{i}+d_{i},\label{eq:sys}
\end{equation}
where $x_{i}$, $x_{i}^{+}\in\mathcal{X}_{i}$ represent the current
and the successive energy states of the DER, respectively, and $d_{i}\in\mathcal{D}_{i}$
is its power consumption. The state and input constraint sets will
be denoted by $\mathcal{X}_{i}=[\underline{x}_{i},\bar{x}_{i}]$ and
$\mathcal{D}_{i}=[\underline{d}_{i},\bar{d}_{i}]$, respectively.
The constant $a_{i}\in(0,1]$ represents the energy dissipation rate.
This model has been widely used to describe the power flexibility
of the HVAC systems (when $a\in(0,1)$) or energy storage (when $a=1$),
see for example, \cite{Hao2015,ZhaoZhangHaoEtAl2017,ZhaoHaoZhang2016}.

To ensure the controllability of the system~(\ref{eq:sys}), we impose
the following conditions
\begin{equation}
a_{i}\underline{x}_{i}+\bar{d}_{i}>\underline{x}_{i},\,a_{i}\bar{x}_{i}+\underline{d}_{i}<\bar{x}_{i}.\label{eq:ctrl}
\end{equation}

Condition~(\ref{eq:ctrl}) guarantees that the DER can be controlled
into $\mathcal{X}_{i}$ with $d_{i}\in\mathcal{D}_{i}$.

Note that the demand of the DER is changing dynamically with the current
states. For example, the DER is less willing to procure power if its
energy state is close to the upper bound. We assume that each DER
submits a bid on the desired power and price to the coordinator in
order to meet its own demand. Given the energy price, the demand of
the DER will be determined by the bidding function. The bidding function
of the DER can be considered as the solution of a payoff maximization
problem defined as follows,
\begin{equation}
\begin{array}{ll}
\mbox{maximize} & v_{i}(d_{i},x_{i})-\lambda d_{i}\\
\mbox{subject to:} & \text{(\ref{eq:sys})},\,x_{i}^{+}\in\mathcal{X}_{i},\,d_{i}\in\mathcal{D}_{i},
\end{array}\label{eq:utility}
\end{equation}
over $d_{i}$, where $v_{i}:\,\,\mathbb{R}\times\mathbb{R}\mapsto\mathbb{R}$
is the consumer's utility function dependent on the current state
$x_{i}$, and $\lambda$ is the energy price. 

Note that the optimal consumption $d_{i}^{*}(\lambda;x_{i})$ is solved
only when $x_{i}\in\mathcal{X}_{i}$. In reality, the DER may not
start with a initial condition that is in $\mathcal{X}_{i}$, then
we simply assume the following consumption policy, 
\begin{equation}
d_{i}=\begin{cases}
\bar{d}_{i}, & \text{if }x_{i}<\underline{x}_{i},\\
\underline{d}_{i}, & \text{if }x_{i}>\bar{x}_{i}.
\end{cases}\label{eq:outside}
\end{equation}
Then under this controllability condition~(\ref{eq:ctrl}), the system
will converge exponentially to $\text{\ensuremath{\mathcal{X}}}_{i}$
from any initial state $x_{i}(0)\in\mathbb{R}\backslash\mathcal{X}_{i}$
under~(\ref{eq:outside}). Therefore, without loss of generality
we will always assume that the initial state lies in $\mathcal{X}_{i}$.

The bidding function $d_{i}^{*}(\lambda;x_{i})$ describes the price
responsiveness of the DER's demand, that is, the change of the demand
with respect to the price. As a general common assumption, we assume
that the utility function $v_{i}$ is a concave function of $d_{i}$.

The coordinator procures energy from electricity providers to meet
the aggregate demand of the DERs. We assume that there is associated
a fixed cost function $c:\,\mathbb{R}\mapsto\mathbb{R}$ with the
providers to supply the energy. As usual, the cost function $c$ is
assumed to be convex, increasing, and $\dot{c}(0)>0$. The providers
are reimbursed at price $\lambda$ and therefore seeks to maximize
its profit by supplying 
\begin{equation}
s^{*}(\lambda)=\arg\underset{s}{\max}\,\lambda s-c(s),\label{eq:cost}
\end{equation}
amount of power. The function $s^{*}(\lambda)$ describes the price
sensitivity of the power supply.

To clear the market, an aggregate demand curve is constructed by the
coordinator from the submitted bidding functions~\cite{li2015}.
It is the inverse mapping of the aggregate demand $\sum_{i\in\mathcal{M}}d_{i}^{*}(\lambda;x_{i})$,
which describes the marginal utility. The supply curve is given by
the optimality condition of~(\ref{eq:cost}), which is $\lambda=\dot{c}(s)$,
i.e., the marginal cost as a function of the supply. Then the market
is cleared at the intersection point $(\lambda_{c},\,s^{*})$ of the
demand curve and the supply curve, where the aggregate demand equals
the supply and the marginal utility equals the marginal cost, see
Fig.~\ref{fig:MC}. This marginal cost $\lambda_{c}$ is usually
referred to as the market clearing price. The tuple $\{\lambda_{c},\,s^{*},\,d_{i}^{*},\forall i\in\mathcal{M}\}$
will be referred to as the market equilibrium.

\subsection{Market Stability}

As discussed in the previous subsection, at each market period, the
DER consumes $d_{i}^{*}(\lambda_{c};x_{i})$ amount of energy. Then
the DER dynamics become
\begin{equation}
x_{i}^{+}=a_{i}x_{i}+d_{i}^{*}(\lambda_{c};x_{i}),\label{eq:GCL}
\end{equation}
for all $i\in\mathcal{M}$. As illustrated in Fig.~\ref{fig:CL},
this is a closed-loop system under the feedback of the market clearing
process. In particular, notice that the market clearing price can
also be viewed as a function of the energy states of the DERs. As
a result, the stability of this closed loop system must be investigated
carefully. Ideally, in the absence of the external disturbances, such
as variation of the electricity cost, coordination signal change,
and weather change, etc., the system~(\ref{eq:GCL}) should converge
to a steady state as fast as possible. In the rest of this paper,
we will establish conditions under which~(\ref{eq:GCL}) is exponentially
stable, that is, there exist an equilibrium $x^{*}\in\mathcal{\mathcal{X}}$,
$M>0$, and $\rho\in(0,1)$ such that for all initial condition $x(0)\in\mathbb{R}^{m}$,
$\forall k\in\mathbb{Z}_{+}$, we have 
\[
\left\Vert x(k)-x^{*}\right\Vert \leq M\rho^{k}\left\Vert x(0)-x^{*}\right\Vert ,
\]
where $x(k)$ is the system states at the $k$th period. Consequently,
the inherent robustness of an exponentially stable system can reduce
the price and power consumption volatility. Hereafter we will refer
to the market stability and the stability of the closed loop system~(\ref{eq:GCL})
interchangeably. 

In the next section, we will first characterize the market equilibrium
at each market period. This enables us to characterize the consumption
profile $d_{i}^{*}(\lambda_{c};x_{i})$ of the DERs. Based on these
characterizations, the stability of the closed-loop system~(\ref{eq:GCL})
is analyzed via contraction analysis.
\begin{center}
\begin{figure*}[!tp]
\centering{}%
\begin{minipage}[t]{0.32\textwidth}%
\begin{center}
\includegraphics[clip,width=0.9\linewidth]{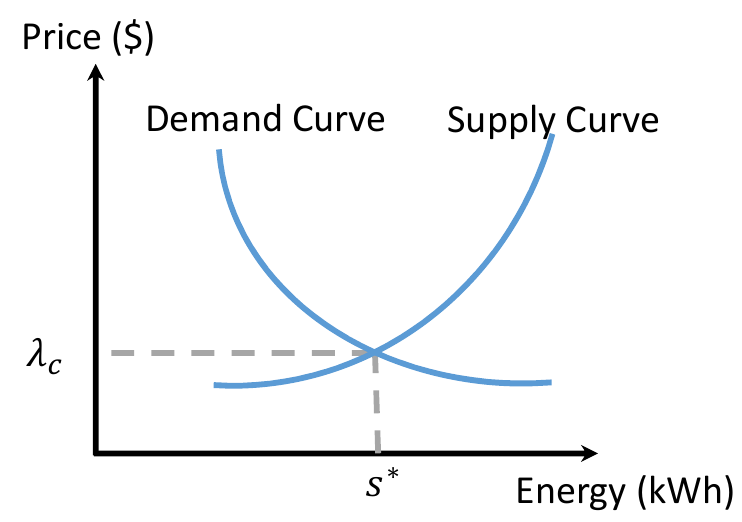}\caption{\label{fig:MC}Typical market clearing process}
\par\end{center}%
\end{minipage}\hfill{}%
\begin{minipage}[t]{0.32\textwidth}%
\begin{center}
\includegraphics[clip,width=0.95\linewidth]{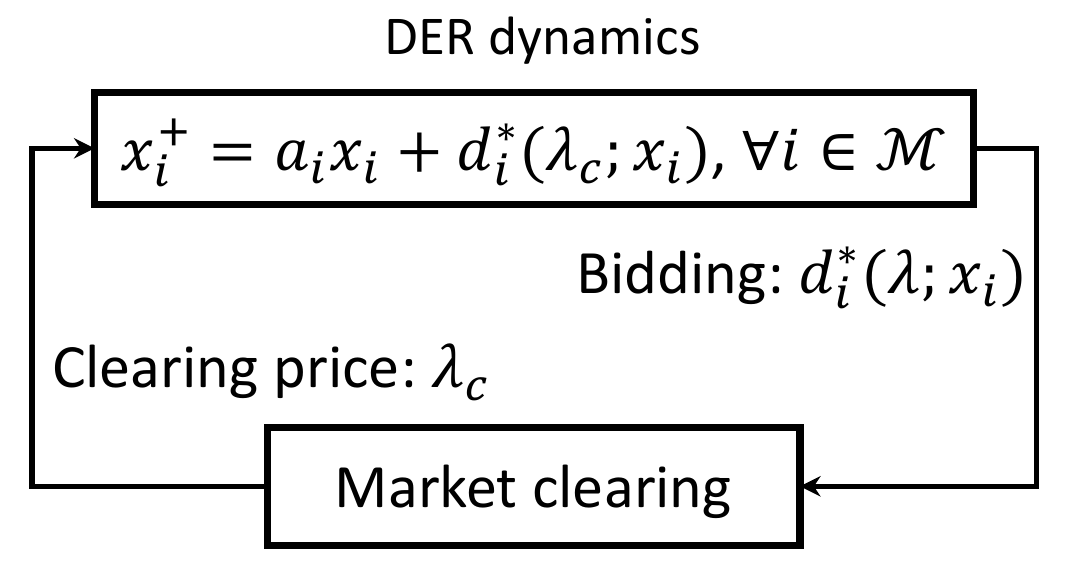}\caption{\label{fig:CL}Closed-loop DER dynamics}
\par\end{center}%
\end{minipage}\hfill{}%
\begin{minipage}[t]{0.32\textwidth}%
\begin{center}
\includegraphics[clip,scale=0.6]{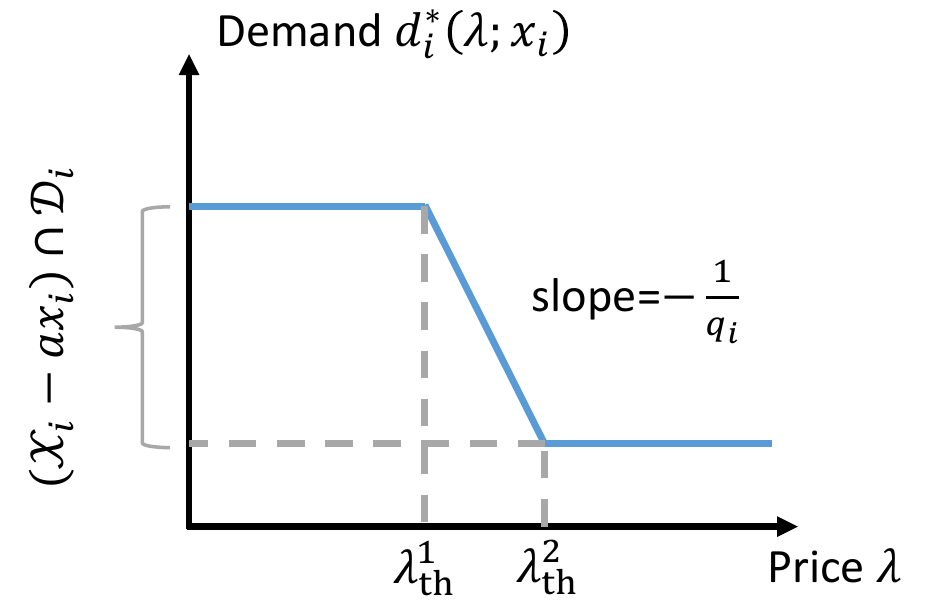}\caption{\label{fig:BidsCurve}Bidding function in~(\ref{eq:bids})}
\par\end{center}%
\end{minipage}
\end{figure*}
\par\end{center}

\section{Stability Analysis}

In this section, we will first define and characterize the market
equilibrium. We show that it can be solved from a social welfare optimization
problem. Under the assumption of quadratic utility and cost functions,
we explicitly characterize~(\ref{eq:GCL}) by a discrete nonlinear
system. We further derive analytical conditions which guarantee the
stability of this nonlinear system. 

\subsection{Competitive Equilibrium}

The intersection point of the demand curve and the supply curve represents
an equilibrium of the market. A \emph{competitive equilibrium }of
the above described electricity market is a tuple $\{\lambda^{*},\,s^{*},\,d_{i}^{*},\forall i\in\mathcal{M}\}$
such that 
\begin{itemize}
\item $d_{i}^{*}$ maximizes the $i$th consumer's payoff, that is, it is
an optimal solution to the problem~(\ref{eq:utility}), for $i\in\mathcal{M}$.
\item $s^{*}$ maximize the profit of the supplier, that is, it is an optimal
solution to the problem (\ref{eq:cost}).
\item $\lambda^{*}$ clears the market, that is, $\sum_{i}d_{i}^{*}(\lambda^{*};x_{i})=c^{-1}(\lambda^{*}).$
\end{itemize}
From the discussion of the last section, it is clear that given the
current state of each DERs, their bids are determined, and the market
is cleared at a competitive equilibrium depending on the DERs' states.
It is well known by the welfare theorems that the competitive equilibrium
is Pareto efficient and every Pareto efficient allocation is attainable
by a competitive equilibrium. The following lemma shows that any competitive
equilibrium is efficient, that is, it is an optimal solution to a
social welfare maximization problem.
\begin{lem}
The competitive equilibrium $\{\lambda^{*},s^{*},d_{i}^{*},\forall i\in\mathcal{M}\}$
of the electricity market is equivalent to the optimal solution of
the following social welfare optimization problem,
\begin{equation}
\begin{array}{ll}
\underset{d_{i},s}{\mbox{maximize}} & \sum_{i=1}^{N}v_{i}(d_{i},x_{i})-c(s)\\
\mbox{subject to:} & x_{i}^{+}=a_{i}x_{i}+d_{i},\\
 & x_{i}^{+}\in\mathcal{X}_{i},\,d_{i}\in\mathcal{D}_{i},\,\forall i\in\mathcal{M}\\
 & \sum_{i=1}^{m}d_{i}=s,
\end{array}\label{eq:social}
\end{equation}
where the market clearing price $\lambda$ emerges as the Lagrangian
dual variable associated with the demand-supply balance constraint
(the last equality constraint in~(\ref{eq:social})). 
\end{lem}
The proof of the above lemma follows the standard argument by checking
that the Karush\textendash Kuhn\textendash Tucker (KKT) conditions
of~(\ref{eq:social}) is equivalent to that of~(\ref{eq:utility})
and~(\ref{eq:cost}), see for example~\cite{HaoCorbinKalsiEtAl2017}. 

Note that the social welfare problem~(\ref{eq:social}) has to be
solved at each market period after the update of the DERs' energy
states. The DER dynamics~(\ref{eq:sys}) with the feedback of the
optimal solution $d_{i}^{*},\,\forall i\in\mathcal{M}$ can be viewed
as a one-horizon model predictive control (MPC) system, or alternatively,
as a system with optimization-based controller~\cite{Primbs2001,HeathWills2005,HeathLi2008}.
It is well known that such systems could be unstable even when the
open loop systems are stable~\cite{Maciejowski2000}. Although various
analytical conditions are proposed in the literature to guarantee
the closed-loop stability of the MPC control system~\cite{MayneRawlingsRaoEtAl2000},
they are generally not applicable to the system considered in this
paper involving the market clearing process~(\ref{eq:social}). In
fact, even the numerical verification of the closed-loop stability~(\ref{eq:GCL})
with quadratic $v_{i}$ and $c$ can be very challenge~\cite{Primbs2001,HeathWills2005,HeathLi2008}.
These works usually assume the knowledge of the equilibrium, which
is usually the origin. In addition, the numerical conditions proposed
there do not scale with the number of the DERs in the market. In the
next subsection, we will work with the quadratic utility and cost
functions and obtain simple analytical conditions to guarantee the
closed-loop system stability and hence the market stability. These
conditions will provide valuable insight into market dynamics as well
as guidance on the design of the bidding functions.

\subsection{Closed-loop DER Dynamics}

We assume that the utility function $v_{i}$ is a general quadratic
function of the power consumption $d_{i}$,
\begin{equation}
v_{i}(d_{i},x_{i})=-\frac{1}{2}q_{i}d_{i}^{2}+(r_{i}x_{i}+c_{i})d_{i},\,\forall i\in\mathcal{M},\label{eq:quad_utility}
\end{equation}
where $q_{i}>0$, $r_{i},\,c_{i}\in\mathbb{R}$ are user-specified
parameters reflecting their preferences. Such quadratic utility functions
have been widely used in coordinating demand-side electric loads via
mean-field game approaches, see for example~\cite{GrammaticoPariseColombinoEtAl2016}
and the references therein. As discussed in the previous section,
the bidding function of the $i$th consumer is given by the optimal
solution of~(\ref{eq:utility}). It can be easily verified that it
is the projection of the optimizer of the unconstrained problem onto
the constrain set, which is
\begin{equation}
d_{i}^{*}(\lambda)=\text{Proj}_{(\mathcal{X}_{i}-a_{i}x_{i})\cap\mathcal{D}_{i}}\left[\frac{-\lambda+r_{i}x_{i}+c_{i}}{q_{i}}\right],\label{eq:bids}
\end{equation}
where the non-emptiness of $(\mathcal{X}_{i}-a_{i}x_{i})\cap\mathcal{D}_{i}$
is guaranteed by the controllability condition~(\ref{eq:ctrl}).

A typical bidding function of the form~(\ref{eq:bids}) is depicted
in Fig.~\ref{fig:BidsCurve}. It contains a linear region in-between
the saturated regions. In the figure, there are two threshold prices
denoted by $\lambda_{\text{th}}^{1}$ and $\lambda_{\text{th}}^{2}$
. Similar bidding function has also been considered in~\cite{li2015}.
Note that the dependence on the load's state models the time-varying
power demand of the consumer. 

The cost function is assumed to be
\begin{equation}
c(s)=\frac{1}{2}\beta_{1}s^{2}+\beta_{2}s,\label{eq:quadcost}
\end{equation}
where $\beta_{i}>0$, $i=1,2$. Then the market price which is given
by the marginal cost is $\lambda=\beta_{1}s+\beta_{2}$.

Using~(\ref{eq:quad_utility}) and~(\ref{eq:quadcost}), we rewrite
the social welfare problem~(\ref{eq:social}) in vector form as follows,
\begin{equation}
\begin{array}{ll}
\underset{d,s}{\mbox{maximize}} & -\frac{1}{2}d^{T}Qd+d^{T}(Rx+c)-\frac{1}{2}\beta_{1}s^{2}-\beta_{2}s\\
\mbox{subject to:} & x^{+}=Ax+d,\,x^{+}\in\mathcal{X},\,d\in\mathcal{D},\\
 & {\bf 1}_{m}^{T}d=s,
\end{array}\label{eq:quad_social}
\end{equation}
where $d,\,x,\,c\in\mathbb{R}^{m}$ and their $i$th components are
$d_{i},\,x_{i},\,c_{i}$, respectively. The diagonal matrices $A,\,Q,\,R$
are generated by the vector $a=\{a_{i}\}$, $q=\{q_{i}\}$, and $r=\{r_{i}\}$
respectively. The state and input constraints are defined by $\mathcal{X}=\Pi_{i=1}^{m}\mathcal{X}_{i}$,
$\mathcal{D}=\Pi_{i=1}^{m}\mathcal{D}_{i}$. 

Without the inequality constraints, the optimal solution to~(\ref{eq:quad_social}),
which we shall refer to as the unconstrained maximizer, can be easily
obtained as 
\begin{equation}
\hat{d}(x)=\tilde{Q}^{-1}(Rx+\tilde{c}),\label{eq:unconsd}
\end{equation}
where $\tilde{Q}=Q+\beta_{1}{\bf 1}_{m}{\bf 1}_{m}^{T}$ and $\tilde{c}=c-\beta_{2}{\bf 1}_{m}$.
The constrained optimal solution can be expressed using $\hat{d}(x)$
conveniently, which is given in the following lemma.
\begin{lem}
The optimal solution to~(\ref{eq:quad_social}) is given by 
\begin{equation}
d^{*}(x)=\text{\emph{Proj}}_{(\mathcal{X}-Ax)\cap\mathcal{D}}^{\tilde{Q}}\hat{d}(x).\label{eq:dproj}
\end{equation}
\end{lem}
\begin{proof}
It follows the standard proof as those in for example~\cite[Lemma 4]{GrammaticoPariseColombinoEtAl2016}. 
\end{proof}
Note that the feedback policy~(\ref{eq:dproj}) is nonlinear due
to the projection. Overall, the closed-loop system becomes 
\begin{equation}
x^{+}=Ax+d^{*}(x).\label{eq:closedloop}
\end{equation}

In particular, it can be shown that~(\ref{eq:closedloop}) is a piecewise
linear system~\cite{BemporadMorariDuaEtAl2002}. We first observe
that the existence of an equilibrium of~(\ref{eq:closedloop}) is
simply guaranteed by the controllability condition~(\ref{eq:ctrl}).
\begin{prop}
\label{prop:uniquess}The closed-loop system~(\ref{eq:closedloop})
has an equilibrium $x^{*}\in\mathcal{X}$ if the controllability condition~(\ref{eq:ctrl})
holds.
\end{prop}
\begin{proof}
The equilibria of $x^{+}=Ax+d$ satisfy $(I-A)x=d$. Set $x^{+}=x$
in~(\ref{eq:quad_social}), and replace $s$ with ${\bf 1}_{m}^{T}d={\bf 1}_{m}^{T}(I-A)x$.
It can be seen that the convex quadratic programming (possibly degenerated)
has a solution if the constraints are non-empty, and the latter can
be found to be exactly the controllability condition~(\ref{eq:ctrl}).
\end{proof}
\begin{rem}
Note that we can further pose conditions to guarantee the uniqueness
of the solution to~(\ref{eq:quad_social}) based on the KKT conditions.
This is particularly for the degenerated case where the quadratic
term is not strongly convex. However, we will not complicate the discussion
here, since our stability results in the sequel will guarantee the
uniqueness of the equilibrium. 
\end{rem}
We will investigate the market stability through analyzing the stability
of the discrete nonlinear system~(\ref{eq:closedloop}). Specifically,
we will find conditions on the utility and the cost functions such
that~(\ref{eq:closedloop}) is exponentially stable. In particular,
this will guarantee no limit cycle exists. This is desired since the
limit cycle corresponds to the oscillation of the power consumption
and the high volatility of the market clearing price. Denote by $x^{+}=T(x)$
the discrete time nonlinear dynamics~(\ref{eq:closedloop}). If we
can show that $T$ is a contraction mapping on $\mathcal{X}$, then
it immediately suggests that there exists a unique equilibrium in
$\mathcal{X}$ such that it is exponentially stable. In the case of
$a_{i}\in(0,1)$, the globally exponential stability can be concluded
in view of the policy~(\ref{eq:outside}). Next, we will first consider
the single DER case, and then extend the results to the multiple DER
case. 

\subsection{Single DER}

This scenario arises frequently when an aggregate DER model is obtained
by the coordinator to facilitate the power planning, see for example~\cite{RoozbehaniDahlehMitter2012,ZhaoHaoZhang2016,ZhaoZhangHaoEtAl2017,Hao2015}.
Such a model represents the aggregate power flexibility of the DER
population. It has the same form of~(\ref{eq:sys}) but with aggregate
power consumption as its input, and has much larger state and input
constraint sets. For clarity, we will denote by (scalar) $a$ the
system matrix $A$ in this section. First, we have the following characterization
of the mapping $T$.
\begin{lem}
\label{lem:ddoubleP}Consider the market consisting of one aggregate
DER model with $m=1$. Then the nonlinear mapping $T$ of the closed-loop
system dynamics~(\ref{eq:closedloop}) is given by 
\begin{equation}
\text{\emph{Proj}}_{\mathcal{X}}\left[ax+\emph{Proj}_{\mathcal{D}}\left[\hat{d}(x)\right]\right],\label{eq:newExp}
\end{equation}
where $\hat{d}(x)$ is defined in~(\ref{eq:unconsd}).
\end{lem}
\begin{proof}
First note that for scalar $\tilde{Q}=q+\beta_{1}>0$, it is easy
to verify that $\text{Proj}_{(\mathcal{X}-ax)\cap\mathcal{D}}^{\tilde{Q}}$
is equivalent to $\text{Proj}_{(\mathcal{X}-ax)\cap\mathcal{D}}$,
i.e., the projection in the $\mathbb{R}^{m}$. Then we need to show
that for arbitrary $d$, the following holds.

\begin{equation}
ax+\text{Proj}_{(\mathcal{X}-ax)\cap\mathcal{D}}d=\text{Proj}_{\mathcal{X}}\left[ax+\text{Proj}_{\mathcal{D}}d\right].\label{eq:doubleProj}
\end{equation}

If $ax+\text{Proj}_{\mathcal{D}}d\in\mathcal{X}$, then $\text{Proj}_{\mathcal{D}}d\in\mathcal{X}-ax$,
and $\text{Proj}_{\mathcal{D}}d=\text{Proj}_{(\mathcal{X}-ax)\cap\mathcal{D}}d$.
Hence, the above equality holds. Otherwise, if $ax+\text{Proj}_{\mathcal{D}}d\notin\mathcal{X}$,
we assume without loss of generality that $ax+\text{Proj}_{\mathcal{D}}d>\bar{x}$.
Then the right hand side of~(\ref{eq:doubleProj}) is $\bar{x}$.
Since $\text{Proj}_{\mathcal{D}}d>\bar{x}-ax$, we have $\text{Proj}_{(\mathcal{X}-ax)\cap\mathcal{D}}d=\bar{x}-ax,$
and therefore the left hand side is also $\bar{x}$. Thus we proved~(\ref{eq:doubleProj})
and also~(\ref{eq:newExp}).
\end{proof}
Under the help of Lemma~\ref{lem:ddoubleP}, we have the following
theorem for the market stability. 
\begin{thm}
\label{thm:OnePlayer}Consider the power market consisting of only
one DER $(m=1)$ and suppose $a\in(0,1]$, $q,\,\beta_{1}>0$, $r\in\mathbb{R}$.
Then the closed-loop system~(\ref{eq:closedloop}) is exponentially
stable in $\mathcal{X}$ if the following condition holds 
\begin{equation}
\left|a+\frac{r}{q+\beta_{1}}\right|<1.\label{eq:onecon}
\end{equation}
\end{thm}
\begin{proof}
By Lemma~\ref{lem:ddoubleP}, the operator $T=\text{Proj}_{\mathcal{\mathcal{X}}}\left[ax+\text{Proj}_{\mathcal{D}}\hat{d}(x)\right]$,
where $\hat{d}(x)=rx/(q+\beta_{1})+\tilde{c}$. Without loss of generality,
assume $x>y$. We then discuss the two cases $r\geq0$ and $r<0$.

Case 1: $r\geq0$. Then $x>y$ implies $\hat{d}(x)\geq\hat{d}(y)$
and it follows that $\text{Proj}_{\mathcal{D}}\hat{d}(x)\geq\text{Proj}_{\mathcal{D}}\hat{d}(y)$,
$ax+\text{Proj}_{\mathcal{D}}\hat{d}(x)\geq ay+\text{Proj}_{\mathcal{D}}\hat{d}(y)$,
and $T(x)\geq T(y)$. Therefore,
\begin{align*}
\left|T(x)-T(y)\right| & =T(x)-T(y)\\
 & \leq a(x-y)+\text{Proj}_{\mathcal{D}}\hat{d}(x)-\text{Proj}_{\mathcal{D}}\hat{d}(y)\\
 & \leq a(x-y)+\hat{d}(x)-\hat{d}(y)\\
 & =(a+\frac{r}{q+\beta_{1}})(x-y)\\
 & <x-y,
\end{align*}
where the first and second inequalities are by the non-expansive property
of the projection operation~~\cite[Proposition 4.8]{BauschkeCombettes2011},
and the last inequality is by condition~(\ref{eq:onecon}).

Case 2: $r<0$. Then $x>y$ implies $\hat{d}(x)<\hat{d}(y)$ and it
follows that $\text{Proj}_{\mathcal{D}}\hat{d}(x)\leq\text{Proj}_{\mathcal{D}}\hat{d}(y)$.
Therefore,
\begin{align*}
\left|T(x)-T(y)\right| & =\left|a(x-y)+\text{Proj}_{\mathcal{D}}\hat{d}(x)-\text{Proj}_{\mathcal{D}}\hat{d}(y)\right|\\
 & \leq\max\Bigl\{\text{Proj}_{\mathcal{D}}\hat{d}(y)-\text{Proj}_{\mathcal{D}}\hat{d}(x)+a(y-x),\\
 & \quad a(x-y)\Bigr\}\\
 & \leq\max\left\{ \hat{d}(y)-\hat{d}(x)+a(y-x),\,a(x-y)\right\} \\
 & =\max\left\{ -a-\frac{r}{q+\beta_{1}},\,a\right\} (x-y)\\
 & <x-y,
\end{align*}
where the second inequality used the non-expansive property of the
projection operation, and the last inequality is by condition~(\ref{eq:onecon}).
Combing these two cases, we see that $T$ is a contraction mapping
on $\mathcal{X}$. Hence, there exists a unique equilibrium in $\mathcal{X}$
which is exponentially stable. Thus we showed that the system is exponentially
stable in $\mathcal{X}$ under condition~(\ref{eq:onecon}). 
\end{proof}
\begin{rem}
The condition~(\ref{eq:onecon}) is exact the sufficient and necessary
stability condition for the unconstrained closed loop system,
\[
x^{+}=Ax+\hat{d}(x).
\]

However, it is not necessary for the constrained closed-loop system~(\ref{eq:closedloop}).
It is easy to give examples for which the systems do not satisfy~(\ref{eq:onecon})
but are still globally exponential stable (with equilibrium on the
boundary point). Also note that for the energy storage DER with $a=1$
(i.e., no energy dissipation), the coupling coefficient $r$ must
be negative.
\end{rem}

\subsection{Multiple DERs}

For the case of multiple DERs, the situation is much more complicated
than the single DER case. The difficulty mainly stems from the weighted
projection $\text{Proj}_{(\mathcal{X}-Ax)\cap\mathcal{D}}^{\tilde{Q}}\hat{d}(x)$
(see~(\ref{eq:dproj})). Even though the projection set $(\mathcal{X}-Ax)\cap\mathcal{D}$
is decoupled in $\mathbb{R}^{m}$, the projection of $\hat{d}(x)$
are coupled among each components since $\tilde{Q}\succ0$ is not
diagonal. This reflects the fact that the optimal power consumption
of each DER will depend on those of other DERs. Clearly, this is a
result of the interaction between DERs through the market coordination.
Due to this coupling, we do not have the characterization of $T$
as in~(\ref{eq:newExp}). However, when the number of DERs is large,
they become weakly coupled. In fact, from the unconstrained maximizer
$\hat{d}(x)$ we have by the Sherman\textendash Morrison formula~\cite{Meyer2000}
that 
\begin{align*}
\hat{d}(x) & =\tilde{Q}^{-1}(Rx+\tilde{c})\\
 & =\left(I_{m}-\frac{\beta_{1}Q^{-1}{\bf 1}_{m}{\bf 1}_{m}^{T}}{1+\beta_{1}\sum_{i=1}^{m}q_{i}^{-1}}\right)Q^{-1}(Rx+\tilde{c})\\
 & =Q^{-1}Rx-\frac{\beta_{1}Q^{-1}{\bf 1}_{m}}{1+\beta_{1}w_{1}}x_{a}+c_{o},
\end{align*}
where $x_{a}=\sum_{i\in\mathcal{M}}\frac{r_{i}}{q_{i}}x_{i}$ is the
aggregate state, $w_{1}=\sum_{i=1}^{m}q_{i}^{-1}$, and $c_{o}$ is
a constant. Moreover, for the $j$th DER, the unconstrained power
can be expressed as 
\begin{align*}
\hat{d}_{j}(x) & =\frac{r_{j}}{q_{j}}x_{j}-\frac{\beta_{1}}{1+\beta_{1}w_{1}}\frac{x_{a}}{q_{j}}+c_{j}\\
 & =\frac{1}{q_{j}}\left(r_{j}x_{j}-\frac{\beta_{1}}{1+\beta_{1}w_{1}}\sum_{i\in\mathcal{M}}\frac{r_{i}}{q_{i}}x_{i}\right)+c_{j},
\end{align*}
where $c_{j}$ is the $j$th component of $c_{o}$. If $q_{i}^{-1}$
is not diminishing, we have as $m\rightarrow\infty$, $w_{1}\rightarrow\infty$.
Hence, the $j$th DER is only affected by the $i$th DER, $\forall i\neq j$,
by a diminishing factor of $\frac{\beta_{1}}{1+\beta_{1}w_{1}}\frac{r_{i}}{q_{i}}.$
This motivates us to approximate $\hat{d}(x)$ which couples over
all the DERs' state by a decoupled one, given by, 
\begin{equation}
\tilde{d}(x):=\Lambda^{-1}Rx+\tilde{Q}^{-1}\tilde{c},\label{eq:dtilde}
\end{equation}
where $\Lambda$ is diagonal. The diagonal matrix $\Lambda$ will
be chosen such that $\epsilon:=\left\Vert \tilde{Q}^{-1}-\Lambda^{-1}\right\Vert _{2}$
is sufficiently small. The following lemma gives a approximation matrix
$\Lambda$ and derives its approximation error bound.
\begin{lem}
\label{lem:errorBound}Choose the approximation matrix $\Lambda^{-1}$
as
\begin{equation}
\Lambda^{-1}=Q^{-1}-\frac{1}{2}\frac{\beta_{1}w_{2}}{1+\beta_{1}w_{1}}I_{m},\label{eq:Lambda}
\end{equation}
where $w_{j}=\sum_{i=1}^{m}q_{i}^{-j}$, for $j=1,2$. Then the approximation
error is 
\[
\epsilon\leq\frac{1}{2}\frac{\beta_{1}w_{2}}{1+\beta_{1}w_{1}}.
\]
\end{lem}
\begin{proof}
Direct calculation yields 
\[
\left\Vert \tilde{Q}^{-1}-\Lambda^{-1}\right\Vert _{2}=\frac{\beta_{1}}{1+\beta_{1}w_{1}}\left\Vert \frac{1}{2}w_{2}I_{m}-q^{-1}q^{-T}\right\Vert _{2},
\]
where we denote by $q^{-1}$ the column vector $[q_{1}^{-1},...q_{m}^{-1}]^{T}$
and by $q^{-T}$ its transpose. Notice that since $q^{-1}q^{-T}$
is a rank one matrix, it has eigenvalues $\{0,\,w_{2}\}$. It follows
immediately that spectral radius of the symmetric matrix $\frac{1}{2}w_{2}I_{m}-q^{-1}q^{-T}$
is $\frac{1}{2}w_{2}$, and hence $\epsilon\leq\frac{1}{2}\frac{\beta_{1}w_{2}}{1+\beta_{1}w_{1}}$.
This completes the proof.
\end{proof}
Using Lemma~\ref{lem:errorBound}, we can easily obtain an error
bound for the limiting case of an arbitrary number of DERs. The following
corollary also reveals that this error bound can be arbitrarily small
if all $q_{i}'s$ are sufficiently large. 
\begin{cor}
Suppose $q_{i}^{-1}$ is not diminishing. Then the approximation error
$\epsilon$ is bounded by $\frac{1}{2}\max_{i\in\mathcal{M}}q_{i}^{-1}$.
\end{cor}
\begin{proof}
If $q_{i}^{-1}$ is not diminishing, then $w_{1}\rightarrow\infty$
as $m\rightarrow\infty$. It follows from Lemma~\ref{lem:errorBound}
that 
\begin{align*}
\epsilon & \leq\frac{1}{2}\frac{\beta_{1}w_{2}}{1+\beta_{1}w_{1}}\leq\frac{1}{2}\frac{\max_{i\in\mathcal{M}}q_{i}^{-1}\beta_{1}w_{1}}{1+\beta_{1}w_{1}},
\end{align*}
and the right hand side of last inequality goes to $\frac{1}{2}\max_{i\in\mathcal{M}}q_{i}^{-1}$
as $m\rightarrow\infty$. This completes the proof.
\end{proof}
It is worth mentioning that in practice, the coefficient $\beta_{1}$
will be relatively very small compare to $q_{i}$, since it denotes
the quadratic cost of the total supply cost. Hence, the constant $1$
will dominate the denominator and the error $\epsilon$ will be very
small even for a finite number of DERs.

Now the error bounds of using $\tilde{d}(x)$ to approximate $\hat{d}$
can be easily obtained by $\left\Vert \tilde{d}(x)-\hat{d}(x)\right\Vert _{2}\leq\epsilon\left\Vert x\right\Vert \max r_{i}$
after applying the matrix induced norm inequality. Note that since
$x\in(\mathcal{X}-Ax)\cap\mathcal{D}$ is bounded, this error bound
is diminishing as $\epsilon\rightarrow0$ if $r_{i}'$s are bounded. 

Based on the above results, next we will approximate the constrained
maximizer $d^{*}(x)$ in~(\ref{eq:dproj}) using $\tilde{d}(x)$
in~(\ref{eq:dtilde}). To avoid the cluster of notation, let $\Omega(x)$
denote the projection set $(\mathcal{X}-Ax)\cap\mathcal{D}$. Intuitively,
if $\epsilon$ is small, we may approximate $d^{*}(x)$ by $\text{Proj}_{\Omega(x)}^{\Lambda}\tilde{d}(x),$
that is, the projection of $\tilde{d}(x)$ onto $\Omega(x)$ in $\mathcal{H}_{\Lambda}$,
where recall that $\tilde{d}(x)$ is the approximation of $\hat{d}(x)$.
This approximation can be justified from the KKT conditions corresponding
to the projection $d^{*}(x)$. Note that by definition, 
\begin{equation}
d^{*}(x)=\arg\min_{z\in\Omega(x)}\left\Vert z-\tilde{d}(x)\right\Vert _{\tilde{Q}}.\label{eq:defdstar}
\end{equation}

The KKT conditions for the above minimization problem can be found
as,
\[
\begin{cases}
z=\tilde{d}(x)-\tilde{Q}^{-1}L^{T}\mu,\\
Lz=b(x),
\end{cases}
\]
where $Lz=b(x)$ are the active inequality constraints in $z\in\Omega(x)$,
and $\mu$ the associated Lagrangian multipliers. We see that the
optimal projection of $\tilde{d}(x)$ is an affine function of $\mu$
which depends explicitly on $\tilde{Q}^{-1}$. Hence, we will use
the following quantity to approximate $d^{*}(x)$, 
\begin{equation}
\tilde{d}^{*}:=\text{Proj}_{\Omega(x)}^{\Lambda}\tilde{d}(x)=\arg\min_{z\in\Omega(x)}\left\Vert z-\tilde{d}(x)\right\Vert _{\Lambda}.\label{eq:FinalD}
\end{equation}

It is worth mentioning that explicit solutions of the optimal primal
and dual solutions $z$ and $\mu$ of~(\ref{eq:defdstar}) can be
obtained, see for example~\cite{BemporadMorariDuaEtAl2002}. A more
thorough analysis based on these explicit solutions will be our future
work.

Now based on the approximation~(\ref{eq:FinalD}), we will prove
the following sufficient condition to guarantee the market stability.
The proof resembles that of the single DER case. The main point is
that the approximation $\tilde{d}^{*}$ is totally decoupled among
each DERs. 
\begin{thm}
\label{thm:Nplayer}Suppose $a_{i}\in(0,1],\,q_{i}>0,\,r_{i}\in\mathbb{R}$.
The closed-loop system~(\ref{eq:closedloop}) is exponentially stable
if for all $i\in\mathcal{M}$,
\begin{equation}
\left|a_{i}+\phi_{i}r_{i}\right|<1,\label{eq:Ncon}
\end{equation}
where $\phi_{i}$ is the $i$th diagonal element of $\Lambda^{-1}$
in~(\ref{eq:Lambda}), given by 
\[
\phi_{i}=q_{i}^{-1}-\frac{1}{2}\frac{\beta_{1}w_{2}}{1+\beta_{1}w_{1}},
\]
and $w_{j}:=\beta_{1}\sum_{j=1}^{m}q_{j}^{-j}$, for $j=1,2$. 
\end{thm}
\begin{proof}
We will prove that the mapping $T(x):=Ax+\text{Proj}_{\Omega(x)}^{\Lambda}\tilde{d}(x)$
is contractive. 

First note that since $\Lambda$ is diagonal and $\Omega(x)$ is hyper
rectangular in $\mathbb{R}^{m}$, it can be easily shown by the definition
of the projection that $\text{Proj}_{\Omega(x)}^{\Lambda}\tilde{d}(x)=\text{Proj}_{\Omega(x)}\tilde{d}(x)$.
It then follows from the same argument of Lemma~\ref{lem:ddoubleP}
component-wisely that 
\begin{equation}
Ax+\text{Proj}_{\Omega(x)}\tilde{d}(x)=\text{Proj}_{\mathcal{X}}\left[Ax+\text{Proj}_{\mathcal{D}}\tilde{d}(x)\right].\label{eq:imm}
\end{equation}

Using~(\ref{eq:imm}) and the non-expansive property of the projection
operation~\cite[Proposition 4.8]{BauschkeCombettes2011}, we have
for arbitrary $x,\,y\in\mathcal{X}$, $x\neq y$, 
\begin{align*}
\left\Vert T(x)-T(y)\right\Vert _{2} & \leq\left\Vert A(x-y)+\text{Proj}_{\mathcal{D}}\tilde{d}(x)-\text{Proj}_{\mathcal{D}}\tilde{d}(x)\right\Vert _{2}.
\end{align*}

Now consider the $i$th element of the vector inside the norm on the
right hand side of the above inequality. Similar to the proof of Theorem~\ref{thm:OnePlayer},
without loss of generality, we assume $x_{i}>y_{i}$. If $r_{i}\geq0$,
we have by the non-expansiveness of the projection that,
\begin{multline}
0\leq a_{i}(x_{i}-y_{i})+\text{Proj}_{\mathcal{D}_{i}}\tilde{d}_{i}(x)-\text{Proj}_{\mathcal{D}_{i}}\tilde{d}_{i}(y)\\
\leq(a_{i}+\phi_{i}r_{i})(x_{i}-y_{i}),\label{eq:iterm1}
\end{multline}
where $\tilde{d}_{i}(x)$ is the $i$th element of $\tilde{d}(x)$
in~(\ref{eq:dtilde}). 

If $r_{i}<0$, then 
\begin{multline}
\left|a_{i}(x_{i}-y_{i})+\text{Proj}_{\mathcal{D}_{i}}\tilde{d}(x)-\text{Proj}_{\mathcal{D}_{i}}\tilde{d}(y)\right|\\
\leq\max\Bigl\{\text{Proj}_{\mathcal{D}_{i}}\tilde{d}_{i}(y)-\text{Proj}_{\mathcal{D}_{i}}\tilde{d}_{i}(x)+a(y_{i}-x_{i}),\,a_{i}(x_{i}-y_{i})\Bigr\}\\
\leq\max\left\{ \tilde{d}_{i}(y)-\tilde{d}_{i}(x)+a(y_{i}-x_{i}),\,a_{i}(x_{i}-y_{i})\right\} \\
=\max\left\{ -a_{i}-\phi_{i}r_{i},\,a_{i}\right\} (x_{i}-y_{i}),\label{eq:iterm2}
\end{multline}

Combining $(\ref{eq:iterm1})$ and (\ref{eq:iterm2}), we have
\begin{align*}
\left\Vert T(x)-T(y)\right\Vert _{2} & \leq\left\Vert \Xi(x-y)\right\Vert _{2},
\end{align*}
where $\Xi$ is a diagonal matrix whose $i$th element is given by
\[
\begin{cases}
a_{i}+\phi_{i}r_{i}, & \text{if \ensuremath{r_{i}\geq0},}\\
\max\left\{ -a_{i}-\phi_{i}r_{i},\,a_{i}\right\} , & \text{if }\ensuremath{r_{i}<0.}
\end{cases}
\]

Then by the assumption~(\ref{eq:Ncon}), 
\[
\left\Vert T(x)-T(y)\right\Vert _{2}\leq\left\Vert \Xi\right\Vert _{2}\left\Vert x-y\right\Vert _{2}<\left\Vert x-y\right\Vert _{2},
\]
which shows that $T$ is a contraction mapping on $\mathcal{X}$.
Hence, the closed-loop system is exponentially stable. Thus it completes
the proof.
\end{proof}
We see from the above derived condition that if $q_{i}^{-1}$ or $r_{i}$
is sufficiently small, we can always guarantee the stability of the
closed-loop system~(\ref{eq:closedloop}). The small $q_{i}^{-1}$
corresponds to a bidding function with a relatively shallower slope.
\begin{center}
\begin{table}
\caption{\label{tab:SimPara}Parameters for Unstable Market Dynamics}

\centering{}%
\begin{tabular}{ccc|ccc}
\multicolumn{1}{c}{} & \multicolumn{2}{c}{} & \multicolumn{1}{c}{} & \multicolumn{2}{c}{}\tabularnewline
\hline 
\hline 
\multirow{2}{*}{Param.} & \multicolumn{2}{c|}{Value} & \multirow{2}{*}{Param.} & \multicolumn{2}{c}{Value}\tabularnewline
\cline{2-3} \cline{5-6} 
 & $m=1$ & $m=100$ &  & $m=1$ & $m=100$\tabularnewline
\hline 
$a_{i}$ & $0.95$ & U$[0.95,\,0.9]$ & $x_{i}(0)$ & U$[\underline{x}_{i},\,\bar{x}_{i}]$ & U$[\underline{x}_{i},\,\bar{x}_{i}]$\tabularnewline
$x_{i}^{r}$ & - & U$[350,\,500]$ & $q_{i}$ & $0.005$ & $0.005$\tabularnewline
$\underline{x}_{i}$ & $2500$ & $x_{i}^{r}-200$ & $r_{i}$ & $-0.1a_{i}$ & $-2a_{i}$\tabularnewline
$\bar{x}_{i}$ & 7500 & $x_{i}^{r}+200$ & $c_{i}$ & 500 & $2x_{i}^{r}$\tabularnewline
$\underline{d}_{i}$ & $0$ & $0$ & $\beta_{1}$ & $0.04$ & $0.008$\tabularnewline
$\bar{d}_{i}$ & $500$ & U$[100,150]$ & $\beta_{2}$ & $20$ & $20$\tabularnewline
\hline 
\hline 
 &  & \multicolumn{1}{c}{} &  &  & \tabularnewline
\end{tabular}
\end{table}
\par\end{center}

\begin{center}
\begin{figure*}[!tp]
\begin{centering}
\subfloat[\label{subfig:Unstable1DPrice}Unstable price response]{\centering{}\includegraphics[bb=30bp 0bp 432bp 176bp,clip,width=0.45\linewidth]{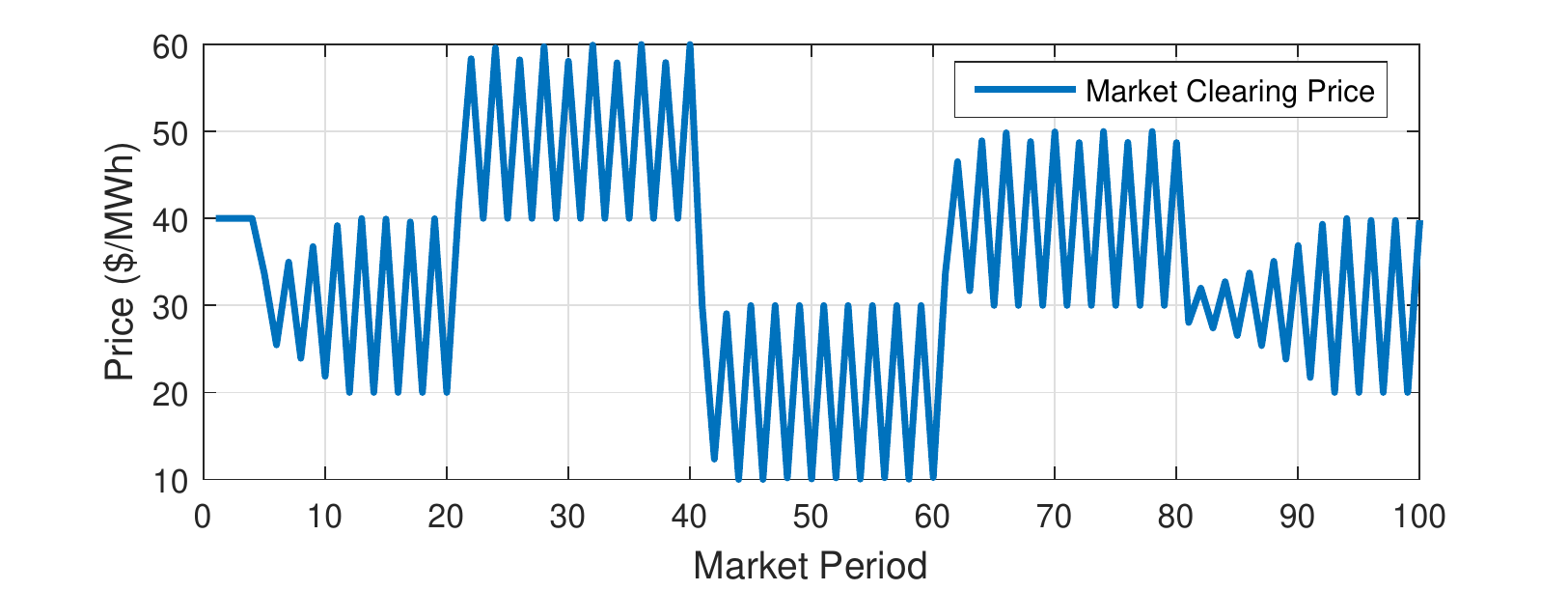}}\hfill{}\subfloat[\label{subfig:Unstable1DPower}Unstable aggregate power response]{\centering{}\includegraphics[bb=24bp 0bp 432bp 176bp,clip,width=0.46\linewidth]{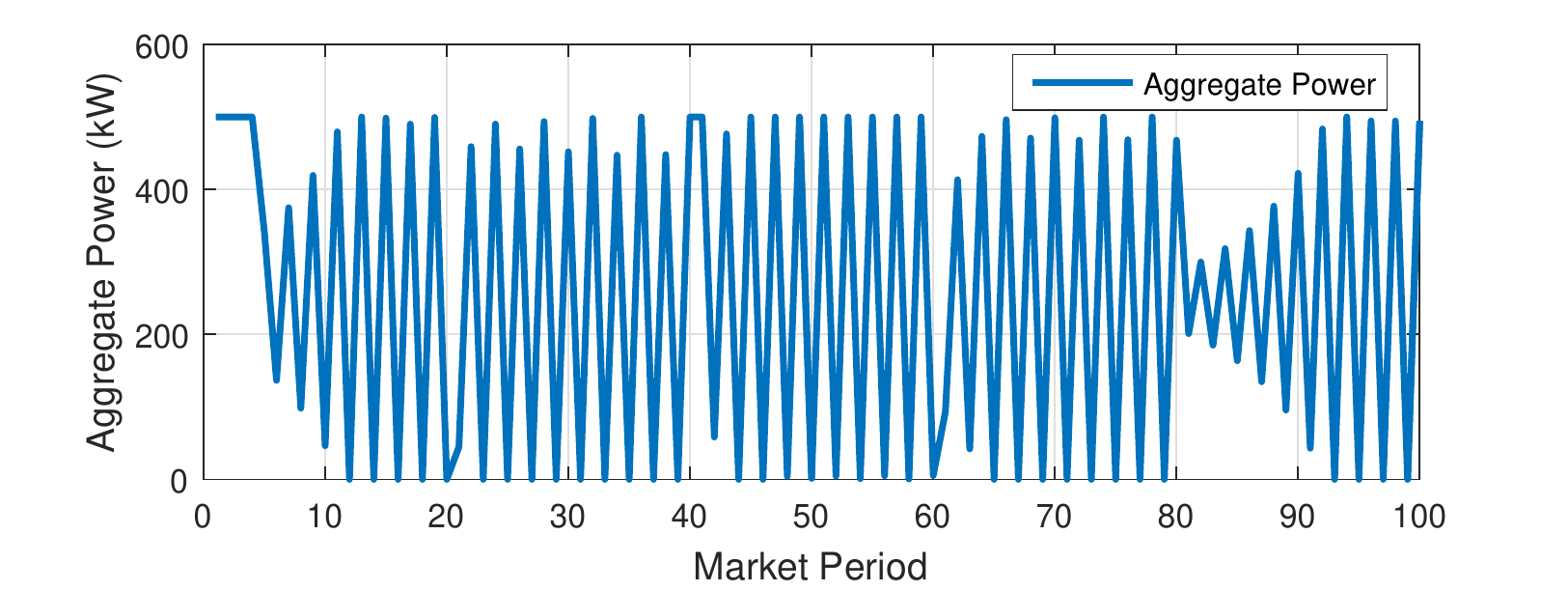}}
\par\end{centering}
\subfloat[\label{subfig:Stable1DPrice}Stable price response]{\centering{}\includegraphics[bb=28bp 0bp 432bp 176bp,clip,width=0.45\linewidth]{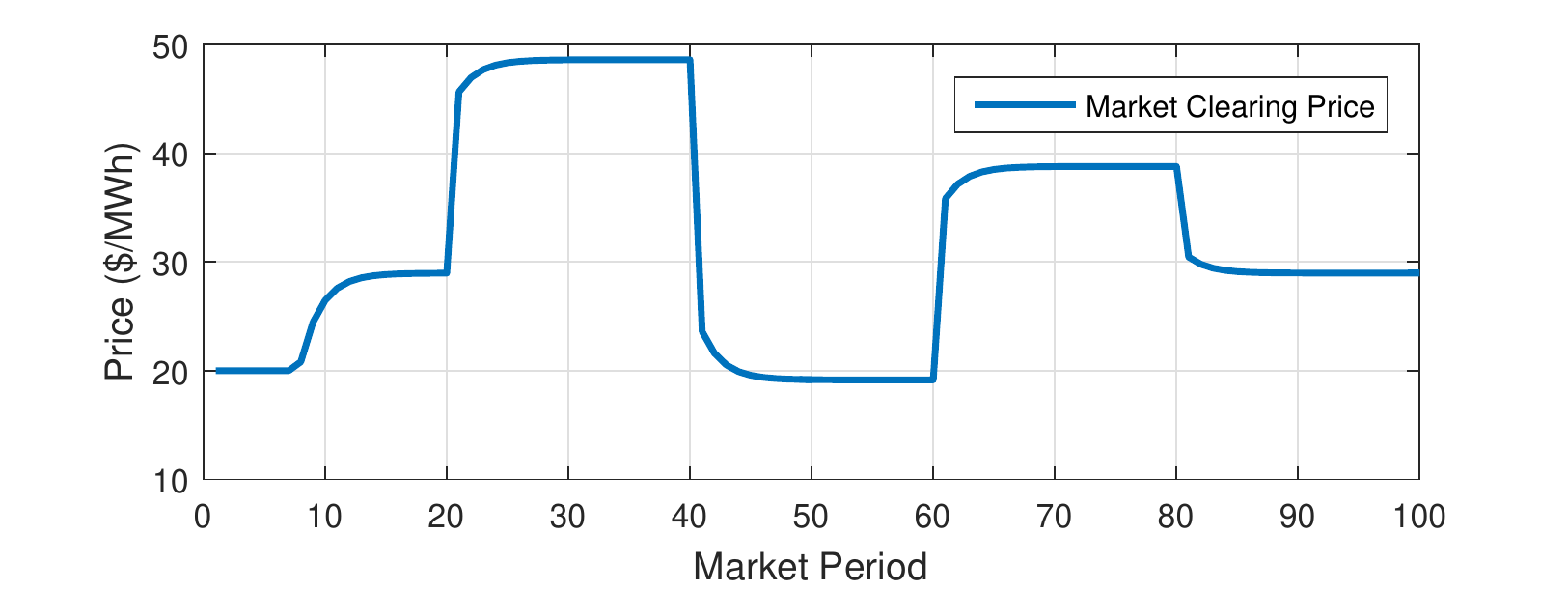}}\hfill{}\subfloat[\label{subfig:Stable1DPower}Stable aggregate power response]{\centering{}\includegraphics[bb=25bp 0bp 432bp 176bp,clip,width=0.45\linewidth]{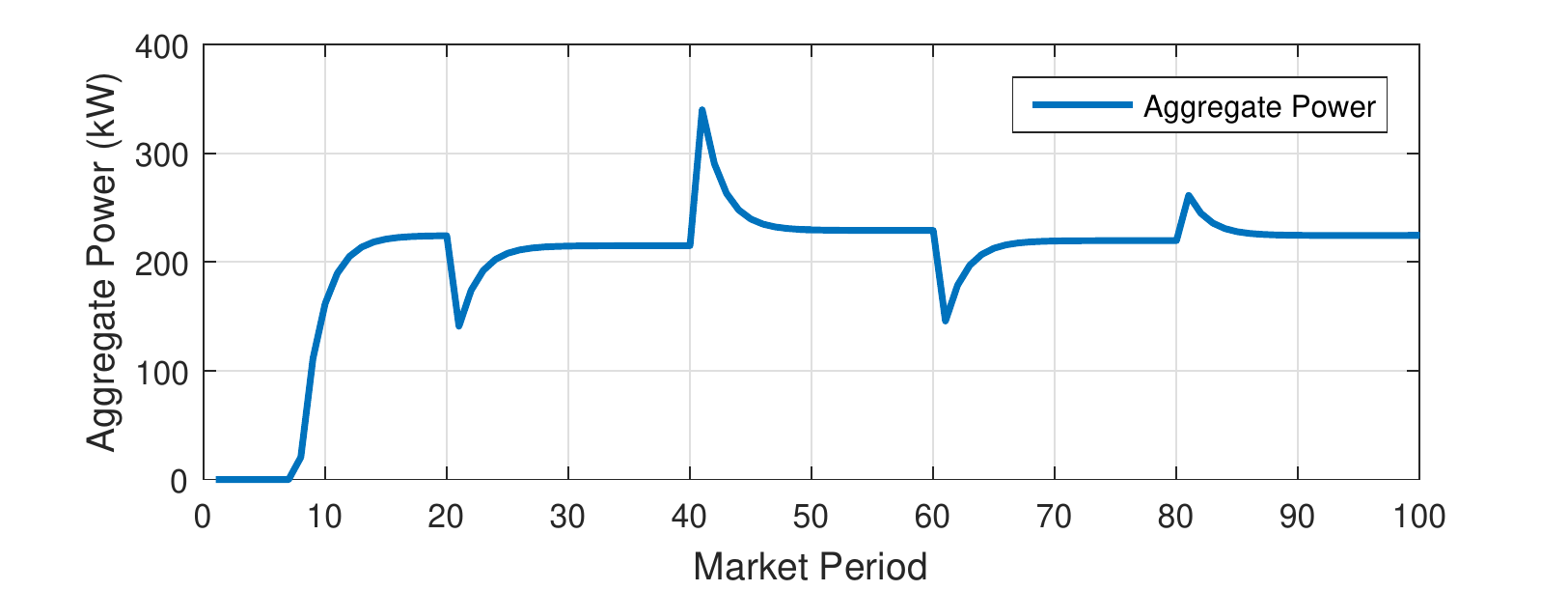}}

\caption{\label{fig:SingleDER}Single DER market under base price change}

\begin{centering}
\subfloat[\label{subfig:UnStableNDPrice}Unstable price response]{\begin{centering}
\includegraphics[bb=30bp 0bp 432bp 176bp,clip,width=0.45\linewidth]{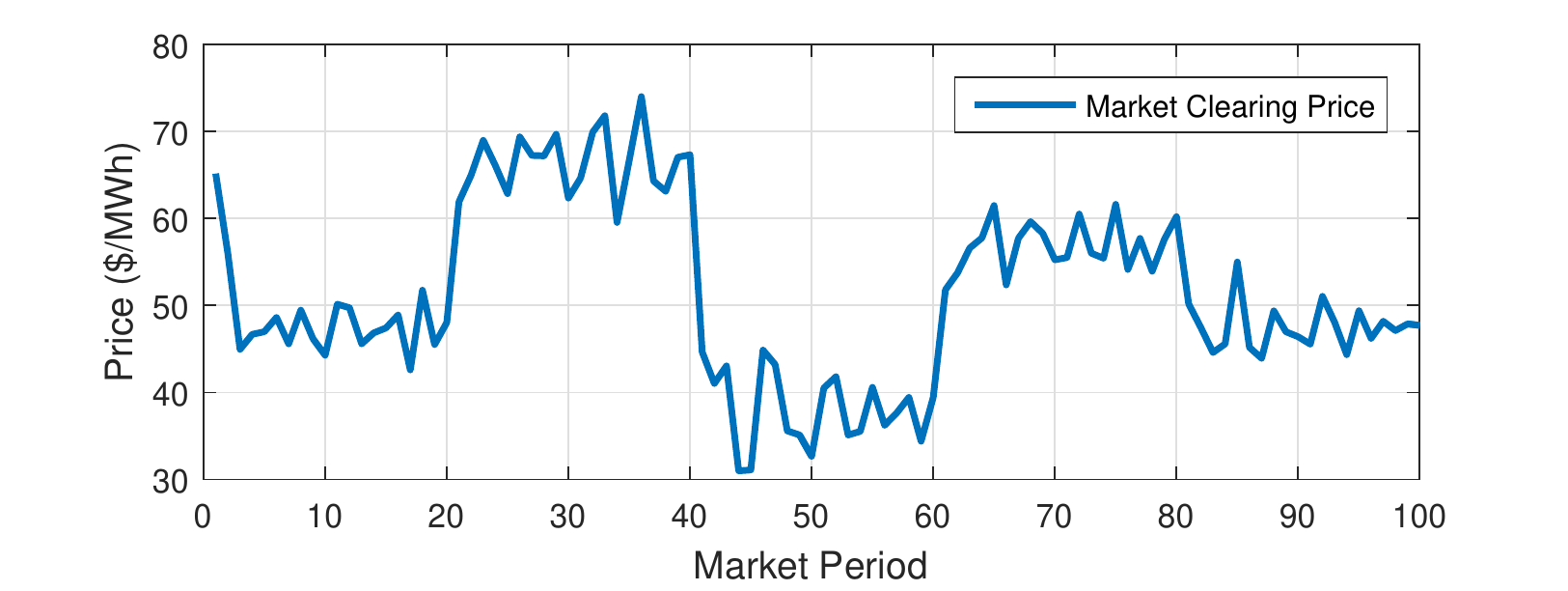}
\par\end{centering}
\centering{}}\hfill{}\subfloat[\label{subfig:UnStableNDPower}Unstable aggregate power response]{\centering{}\includegraphics[bb=35bp 0bp 432bp 176bp,clip,width=0.45\linewidth]{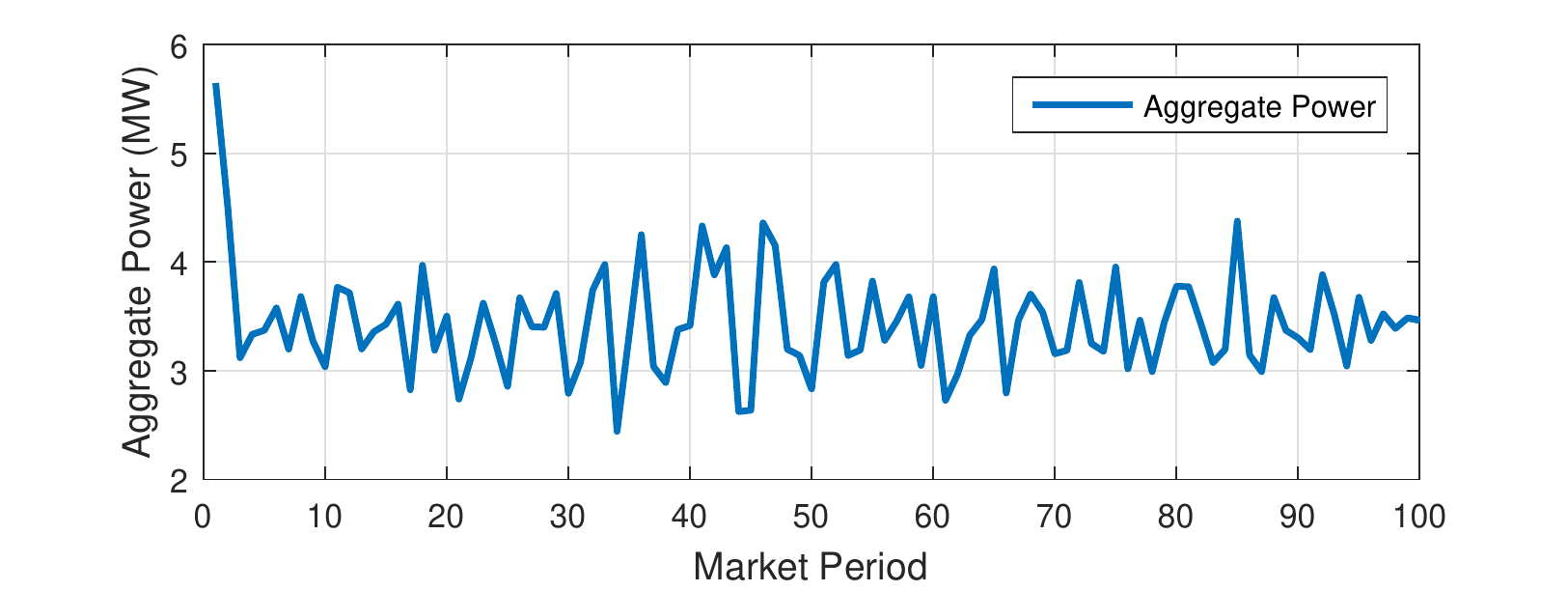}}
\par\end{centering}
\subfloat[\label{subfig:StableNDPrice}Stable price response]{\centering{}\includegraphics[bb=30bp 0bp 432bp 176bp,clip,width=0.45\linewidth]{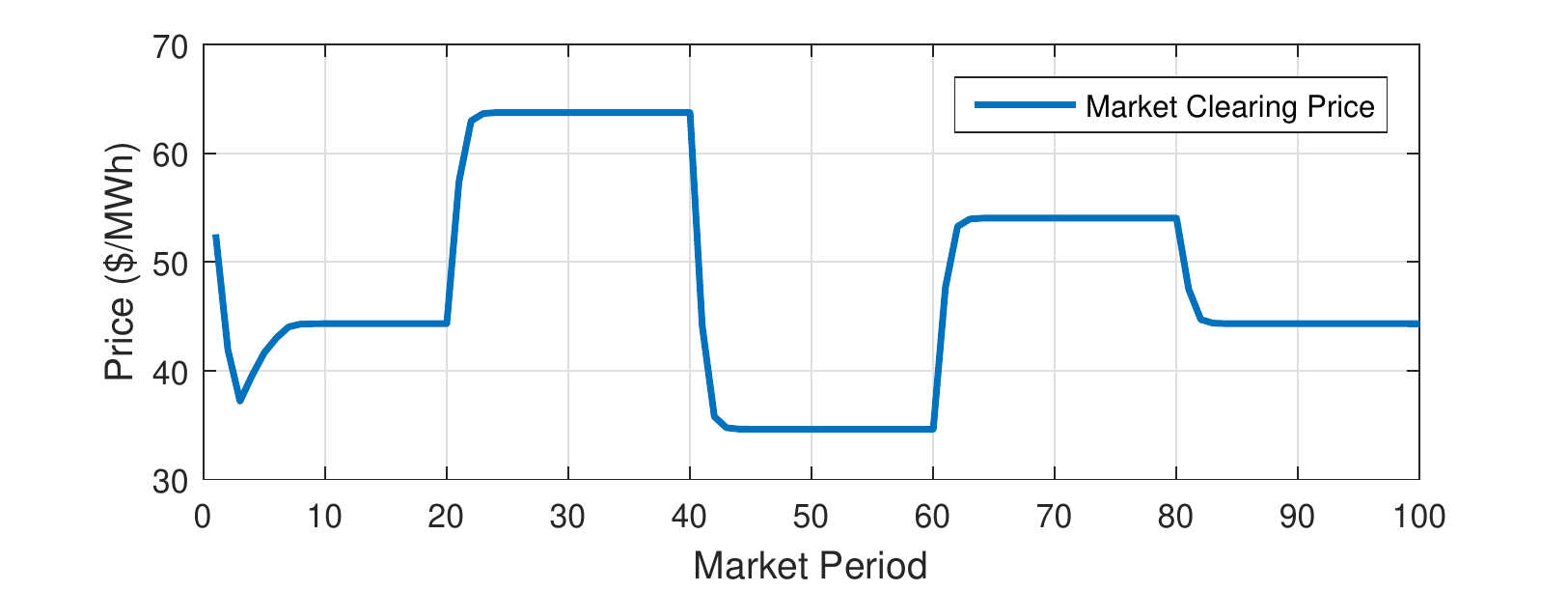}}\hfill{}\subfloat[\label{subfig:StableNDPower}Stable aggregate power response]{\centering{}\includegraphics[bb=27bp 0bp 432bp 176bp,clip,width=0.45\linewidth]{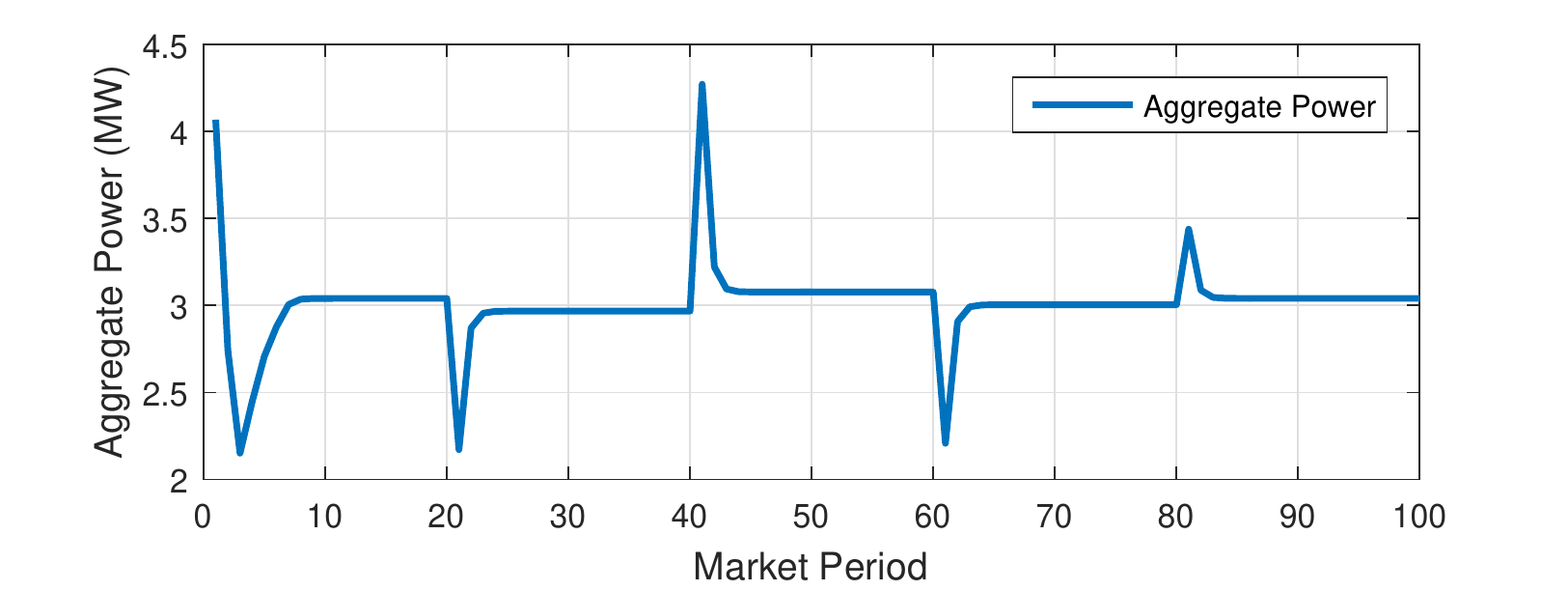}}

\caption{\label{fig:MultiDER}Multi-DER market under base price change}
\end{figure*}
\par\end{center}

\section{Numerical Examples}

In this section, we will present some examples to demonstrate the
application of the derived stability conditions. We consider a scenario
where the base price changes over time. Note that the base price corresponds
to the constant term $\beta_{2}$ in the marginal cost $\lambda=\beta_{1}s+\beta_{2},$
which can be influenced by the external wholesale market price.

\subsection{Single DER}

We first consider the market consists of only one aggregate DER. The
columns below $m=1$ in Table~\ref{tab:SimPara} give the simulation
parameters for the\emph{ unstable} market. In the simulation, the
base price $\beta_{2}$ is set to $\{20,40,10,30,20\}$ successively
every 20 market periods. The price and the aggregate power evolution
are depicted in Figs.~\ref{fig:SingleDER}\ref{subfig:Unstable1DPrice}-\ref{subfig:Unstable1DPower},
respectively. We can see that the resulting market clearing prices
and aggregate power keep oscillating. Moreover, the base price as
a coordinating signal fails to adjust the aggregate power consumption
effectively. Note that in this case, we have $q=0.005$ and the ratio
$a+\frac{r}{q+\beta_{1}}=-1.1611$, which violates the condition~(\ref{eq:onecon}). 

If we increase $q$ to $0.2$, then $a+\frac{r}{q+\beta_{1}}=0.5542$
and the condition~(\ref{eq:onecon}) is satisfied. The resulting
market clearing prices and aggregate power become stable in this case,
as can be seen from in Figs.~\ref{fig:SingleDER}\ref{subfig:Stable1DPrice}-\ref{subfig:Stable1DPower}.

\subsection{Multiple DERs}

We next assume that there are $100$ DERs modeled by~(\ref{eq:sys}).
The parameters for the system dynamics, utility functions, and the
cost functions are generated randomly according to the distribution/values
in the $m=100$ columns of Table~\ref{tab:SimPara}, where U$[h,l]$
denotes the uniform distribution within the range $[h,l]$, and the
variable $x_{i}^{r}$ is used to generate several other variables
in the table. The same base price change scenario in the single DER
case is simulated where $\beta_{2}$ is set to $\{20,40,10,30,20\}$
successively every 20 market periods. We solve the social welfare
problem~(\ref{eq:quad_social}) to obtain the market clearing prices
and energy allocation at each market period. As shown in Figs.~\ref{fig:MultiDER}\ref{subfig:UnStableNDPrice}-\ref{subfig:UnStableNDPower},
both the market clearing prices and the aggregate power are highly
volatile and oscillating with large amplitudes. In fact, the condition~(\ref{eq:Ncon})
are violated for all DERs with $a_{i}+\phi_{i}r_{i}\in(-190,-180)$.
We next keep the other parameters the same, and choose $q_{i}=1.5$
in order to stabilize the market. With this new $q_{i}$, the condition~(\ref{eq:Ncon})
is satisfied for all DERs, where $a_{i}+\phi_{i}r_{i}$'s are around
$-0.09$. The resulted market clearing price and aggregate demand
evolution are illustrated in Figs.~\ref{fig:MultiDER}\ref{subfig:StableNDPrice}-\ref{subfig:StableNDPower}.
We can see that the market converges to an equilibrium very quickly
within 10 market periods. It is worth mentioning that one could also
change $r_{i}$ to stabilize the market following the stability condition~(\ref{eq:Ncon}).
Our extensive simulation shows that condition~(\ref{eq:Ncon}) is
a very efficient certification for the market stability under multiple
heterogeneous DERs.

\section{Conclusions and Future Work}

This paper investigated the electricity market stability under dynamic
DER models. The individual DER was modeled by a scalar linear system
with both state and input constraints. Under the assumption of quadratic
utility and cost functions, we characterized the competitive equilibrium
of the market, and convert the market stability into a discrete nonlinear
system stability problem. The stability analysis of such systems is
very challenging in general. We derived analytical conditions to guarantee
the system stability via a contraction analysis approach. These conditions
implied that the market stability can be guaranteed by simply choosing
shallower slope or smaller coupling coefficient between individual
state and consumption. Numerical examples were provided to demonstrate
the application of the stability results.

Our future work includes further investigation of less conservative
conditions for the market stability and incorporating the feeder capacity
limits into the market model.

\bibliographystyle{IEEEtranS}
\bibliography{CDC2017}

\end{document}